\documentclass[11pt]{amsart}
\usepackage{amsfonts,amssymb}
\usepackage{amsmath,amscd}
\usepackage[all]{xy}
\usepackage{color}

 \newtheorem{thm}{Theorem}
 \newtheorem{cor}{Corollary}
 \newtheorem{lem}{Lemma}
 \newtheorem{prop}{Proposition}

 \theoremstyle{definition}
 \newtheorem{defn}{Definition}
 \newtheorem{ex}{Example}
 \newtheorem{rem}{Remark}

 \newcommand{\nc}{\newcommand}
 \newcommand{\Hom}{\operatorname{Hom}}
 \newcommand{\End}{\operatorname{End}}
 \newcommand{\Ext}{\operatorname{Ext}}
 
 \newcommand{\Ind}{\operatorname{Ind}}
\newcommand{\Res}{\operatorname{Res}}

\newcommand{\mbb}{\mathbb}
\newcommand{\mcal}{\mathcal}
\newcommand{\mbf}{\mathbf}
\newcommand{\ul}{\underline}

\nc{\ula}{\ul{\alpha}}
\nc{\ui}{\ul{i}}
\nc{\uk}{\ul{k}}
\nc{\uik}{\ui,\uk}
\nc{\wt}{\widetilde}

\nc{\redtext}[1]{\textcolor{red}{#1}}
\nc{\zb}[1]{\redtext{#1}}

\begin{document}
 \title{Geometric approach to Hall algebra of representations of Quivers over local rings}
 \author{Zhaobing Fan}

  \maketitle

\begin{abstract}
By using  perverse sheaves on representation spaces of quivers over  $k[t]/(t^n)$ and
 jet schemes over  flag varieties,
we construct a geometric composition algebra $\mbf K$ under Lusztig's framework on geometric realizations of the negative part of quantum algebras.
Simple perverse sheaves in $\mbf K$ form the canonical basis of $\mbf K$.
The relationships among the algebra $\mbf K$, the composition algebra of locally projective representations of quivers over $k[t]/(t^n)$ and
quantum generalized Kac-Moody algebra are provided.
 \end{abstract}

\setcounter{tocdepth}{1}
\tableofcontents

 \section{Introduction}

The canonical basis theory of quantum algebras is vitally important in Lie theory.
This basis has many remarkable properties such as integrality and positivity of
structure constants.
The canonical basis of the negative part of quantum algebras can be
naturally obtained via their geometric realizations.
The canonical basis is constructed by Lusztig in \cite{L1} for finite types,
as well as by  Lusztig, Kang and Schiffmann etc. in \cite{L2,L5, KS} for other types
via  geometric realizations of the negative/positive part of quantum algebras.

In the geometric realization of the negative part of quantum algebras,
flag varieties over an algebraically closed field $k$ are frequently used.
In the present paper,
we consider jet schemes over these flag varieties.
More precisely, we consider the varieties of filtrations of projective $R$-modules,
where $R=k[t]/(t^n)$.
Under Lusztig's framework, we construct an algebra $\mbf K$, which is called the geometric  composition algebra over $R$.
The algebraic correspondence to this geometric setting is locally projective representations (see definition in section \ref{sec4.2}) of
 quivers over $R$.
The category of locally projective representations of
 quivers over $R$ is not an abelian category anymore, but
rather an exact category.
In literature, Hall algebras over exact categories are studied in \cite{H1, KSV} etc.
The subalgebra of this Hall algebra generated by irreducible locally projective representations is then called the algebraic composition algebra.
There is an algebra homomorphism from the geometric composition algebra to the algebraic one.
Simple perverse sheaves in $\mbf K$ form the canonical basis of $\mbf K$.
It will be interesting to describe canonical basis elements analogous to \cite{Li} by Li, via our algebraic setting.

From a geometric point of view,
jet schemes over flag varieties themselves are of particular interest.
They are vector bundles over ordinary flag varieties.
From an algebraic point of view,
a locally projective representation of quivers over $R$ is a deformation of a representation of quivers over $k$.
Representations of quivers over rings can be traced back to \cite{LZ}.
More generally, one could consider  representations of quivers over an algebraic variety $X$,
 namely, assigning a vector bundle on $X$ to each vertex and
 a bundle map to each arrow.
 The obtained category is still an exact category.
 This issue will be addressed in a forthcoming paper.
 What we considere here is a special case of it.

In the context we are considering,
the first technical barrier is that the jet scheme over
flag varieties is
not a projective variety.
This causes the projection map along the
jet scheme to representation spaces not being a proper map. Hence, the decomposition
theorem in \cite{BBD} cannot be directly applied here.
However,
the projection map is still good enough because we can decompose the
projection map into a composition of a proper map and a vector
bundle.
One could consider representations of quivers over arbitrary commutative rings,
but the corresponding geometry does not always have this nice property.
The second technical barrier is that Lusztig's restriction functor is not well-defined in this setting because
submodules and quotient modules of projective $R$-modules are not necessarily projective.
To tackle this problem, we replace Lusztig's restriction functor by its free part after showing that
it can be decomposed into the sum of the free part
and the torsion part.
On the algebraic side,  it will be interesting to construct a coalgebra structure and the antipode
 for the algebraic composition algebra analogous to \cite{Gr, X1} by Green and Xiao, respectively.

 The category of representations of a quiver over $R$ is a subcategory of representations of a new quiver,
 obtained by adding loops on each vertex, over field $k$.
 The simplest example is the quiver
$\Gamma$ with one vertex and no arrows.
In this example, the
category of representations of the quiver $\Gamma$ over
$R$ is the category of $R$-modules. This category is
equivalent to the category of nilpotent representations of the
Jordan quiver over $k$.
In \cite{KS}, Kang and Schiffmann give a
geometric realization of the positive part of  quantum generalized Kac-Moody algebras by using
quivers with multiple loops,
whose canonical bases are given by semisimple perverse sheaves.
Therefore, one may expect to obtain Kang and
Schiffimann's result through geometry of representations of quivers over the
local ring $R$.
In fact, there is an algebra homomorphism from the quantum
generalized Kac-Moody algebra with  charges are all equal to one to the geometric composition algebra.
Moreover, the canonical basis of the geometric composition algebra are given by simple perverse
sheaves.

 This paper is organized as follows. In Section $\rm 2$, we
briefly review
 the theory of perverse
sheaves.
 In Section $\rm 3$, we study the geometric property of jet scheme over ordinary flag
varieties and perverse sheaves on the locally projective representation
spaces over $R$.
We  construct the geometric composition algebra,
whose canonical basis and monomial basis are provided.
In Section 4, we study the Hall algebra over the category of locally projective representations of quivers over $R$,
and the relationship between the geometric composition algebra and the algebraic one.
In Section 5, the relationship between quantum generalized Kac-Moody algebras
and the geometric composition algebra is provided. \\

{\bf Acknowledgments}
  I would like to thank  Zongzhu Lin for his guidance and
  help. I also would like to thank
   Yiqiang Li for valuable comments and bringing my attention to paper \cite{Braden} which simplifies the proof of Lemma \ref{lem3.3}.
  I am grateful to Andrew Hubery for  email communications concerning Hall
  algebras over exact categories. 
  This work 
  is supported by The China Scholarship Council and Harbin
  Engineering University, and is partially supported by the
  Mathematics Department at Kansas State University.


\section{Preliminary}

For the readers' convenience,
we set up notations regarding geometry and perverse sheaves in this section.
We refer to \cite{BBD,FK,L5} for more details.

Throughout this paper,  $k$ is an algebraic closure of $\mathbb{F}_q$.
All algebraic varieties are over $k$ and have an $\mbb F_q$-structure.
Moreover, all algebraic varieties are of finite type and separable.

\subsection{Perverse sheaves}

Let $X$ be an algebraic variety. Denote by
$\mathcal{D}(X)$ the bounded derived category of
$\overline{\mathbb{Q}}_l$-constructible sheaves on $X$.
 Here $l$ is a fixed
prime number which is invertible in $k$, and
$\overline{\mathbb{Q}}_l$ is an algebraic closure of the field
$\mathbb{Q}_l$ of $l$-adic numbers.

 Let $\mathcal{M}(X)$ be the full subcategory of $\mathcal{D}(X)$ consisting of perverse sheaves on $X$.
The simple objects of $\mathcal{M}(X)$ are given by
the Deligne-Goresky-Macpherson intersection cohomology complexes
corresponding to various smooth irreducible locally closed subvarieties of $X$ and
 irreducible local systems on them.

Let $\mcal D_{\leq w}(X)$ (resp. $\mcal D_{\geq w}(X)$) be the full subcategory of $\mcal D(X)$
consisting of all mixed complexes, whose $i$-th cohomology has weight $\leq w+i$ (resp. $\geq w+i$).
A complex $K$ is called pure of weight $w$ if $K\in \mcal D_{\leq w}(X)\cap \mcal D_{\geq w}(X)$.

We denote by $\mbf 1_X$ the constant sheaf on $X$, and simply write $\mbf 1$ if there is no confusion in context.

\subsection{Functors}
For a complex $K \in \mathcal{D}(X)$,
denote by $\mathcal{H}^n(K)$ the $n$-th cohomology sheaf of $K$.
For any  $j\in \mbb Z$, let $[j]: \mathcal{D}(X) \rightarrow
\mathcal{D}(X)$ be the shift functor which satisfies
$\mathcal{H}^n(K[j])=\mathcal{H}^{n+j}(K)$.

Let $f: X \rightarrow Y$ be a morphism of algebraic varieties. There
are functors $f^*: \mathcal{D}(Y) \rightarrow \mathcal{D}(X), \ f_*:
\mathcal{D}(X) \rightarrow \mathcal{D}(Y), \ f_!: \mathcal{D}(X)
\rightarrow \mathcal{D}(Y)$ (direct image with compact support), and
$f^!: \mathcal{D}(Y) \rightarrow \mathcal{D}(X)$.

Denote by
$\mathbb{D}K$ the Verdier dual of $K \in
\mathcal{D}(X)$ and by ${}^p\!H^n: \mcal D(X)\rightarrow \mcal M(X)$ the perverse
cohomology functor for any $n\in \mbb Z$.

A complex $K \in \mathcal{D}(X)$ is called semisimple if ${}^p\!H^n(K)$ is
semisimple in $\mathcal{M}(X)$ for all $n\in \mbb Z$ and $K$ is isomorphic to
$ \bigoplus_{n\in \mbb Z}{}^p\!H^n(K)[-n]$ in $\mathcal{D}(X)$.

 For any  $n\in \mbb Z$, denote by $\mathcal{M}(X)[n]$ the full
subcategory of $\mathcal{D}(X)$ whose objects are of the form $K[n]$
for some $K \in \mathcal{M}(X)$.
Assume that $f: X \rightarrow Y$ is a locally trivial
principal $G$-bundle. Let $d={\rm dim}(G)$. If $K \in
\mathcal{M}(Y)[n+d]$, then $f^*K \in \mathcal{M}_G(X)[n]$.
Furthermore, the functor $f^*: \mathcal{M}(Y)[n+d] \rightarrow
\mathcal{M}_G(X)[n]$ defines an equivalence of categories. The
inverse $f_{\flat}: \mathcal{M}_G(X)[n] \rightarrow \mathcal{M}(Y)[n+d]$
is given by $f_{\flat}(K)={}^p\!H^{-n-d}(f_*K)[n+d]$.

\subsection{Some Properties}\label{decomposition}

\begin{itemize}
  \item[(a)] Simple perverse sheaves are pure and pure complexes are semisimple.

  \item[(b)]
Let $f: X \rightarrow Y$ be a morphism of varieties with $X$ smooth.
If $f$ is a proper map, then $f_!({\bf 1})$ is a semisimple complex in $\mathcal{D}(Y)$.
More generally, if there is a
partition $X=X_0\cup X_1 \cup \cdots \cup X_m$ of locally closed subvarieties, such that $X_{\leq
j}=X_0 \cup \cdots \cup X_j$ is closed for $j=0,\cdots,m$, and for
each $j$ there are morphisms $X_j \xrightarrow{f_j} Z_j
\xrightarrow{f'_j}Y_j$, such that $Z_j$ is smooth, $f_j$ is a vector
bundle, $f'_j$ is proper and $f'_jf_j=f|_{X_j}$, then $f_!({\bf 1})$
 is a semisimple complex in $\mathcal{D}(Y)$.

\item[(c)] If the following diagram
$$\xymatrix{
X \ar[r]^{f} \ar[d]_r & Y \ar[d]^s \\
 Z \ar[r]^{g} & W} $$
is a cartesian square and $s$ is a proper map (resp. $g$ is a smooth map), then
$$r_!f^*=g^*s_!: \mathcal{D}(Y) \rightarrow \mathcal{D}(Z).$$
\end{itemize}

\subsection{Characteristic functions of complexes}\label{sec2.5}

 Let $F$ be a Frobenius morphism of $X$ and $X^F$ be the set of fixed points by $F$.
 For any complex $\mathcal{F} \in D^b(X^F,\overline{\mathbb{Q}}_l)$,
 such that $F^*\mathcal{F} \simeq \mathcal{F}$,
 we choose an isomorphism $\phi_{\mathcal{F}}$ for each such $\mathcal{F}$.
 The characteristic function of $\mathcal{F}$ with respect to $\phi_{\mathcal{F}}$,
 denote by $\chi_{\mathcal{F},\phi_{\mathcal{F}}}$ can be defined as follows
$$\chi_{\mathcal{F},\phi_{\mathcal{F}}}(x)
=\sum_{i\in \mbb Z}(-1)^iTr(\phi_{\mcal F,x}^i: \mcal H^i(\mathcal{F}_x) \rightarrow \mcal H^i(\mathcal{F}_x)),\  \forall \ x \in X^F,$$
where $\mathcal{F}_x$ is the stalk of $\mathcal{F}$ at $x$.

  If $f: X \rightarrow Y$ is a morphism defined over $\mathbb{F}_q$, and $\mathcal{F} \in D^b(X^F, \overline{\mathbb{Q}}_l)$, then for any $y \in Y^F$,
      $$\chi_{f_!\mathcal{F},\phi_{\mathcal{F}}}(y)=\sum_{x\in f^{-1}(y)^F}\chi_{\mathcal{F},\phi_{\mathcal{F}}}(x).$$

 \section {Lusztig's geometric composition algebra $\mbf K$}

In this section, we fix $R=k[t]/(t^n)$, a loop-free quiver
$\Gamma=(I,H,s,t)$ and
an $I$-graded projective $R$-module $V=\bigoplus_{i \in
I}V_i$.
We note that, in the case of $R=k[t]/(t^n)$,
$V$ is a projective $R$-module if and only if $V$ is a free $R$-module.
Via the ring homomorphism $k \hookrightarrow R$,
an $R$-module $V$ can be thought of as a $k$-vector space.
We shall assume that $V_i$ is a finite dimension vector space for all $i\in I$.
For simplicity, we write $h'=s(h)$ and $h''=t(h)$ for all $h\in H$.
We define
\begin{equation}\label{eq3.6}
  E_V^k=\bigoplus_{h \in H} \Hom_k(V_{h'}, V_{h''}),\quad  E_V^R=\bigoplus_{h \in H} \Hom_R(V_{h'}, V_{h''}),
\end{equation}
\begin{equation}\label{eq3.9}
  G_V^k=\bigoplus_{i\in I} GL_k(V_i),\quad G_V^R=\bigoplus_{i\in I} GL_R(V_i),
\end{equation}
where $GL_k(V_i)$ (resp. $GL_R(V_i)$) is the set of all $k$-automorphisms (resp. $R$-automorphisms) of $V_i$.
The algebraic group $G_V^R$ (resp. $G_V^k$) acts on $E_V^R$ (resp. $E_V^k$) by
conjugation, i.e., $gx=x'$ and $x'_h=g_{h''}x_hg^{-1}_{h'}$ for all
$h \in H$.

Given $R$-modules $V_1$ and $V_2$, $\Hom_k(V_1,V_2)$ has an $R$-module structure defined by
$$(r f)(v)=f(rv)-rf(v),$$
for all $r \in R$, $v \in V_1$ and $f \in \Hom_k(V_1,V_2)$. Then
$$\Hom_R(V_1,V_2)=\left\{f \in \Hom_k(V_1,V_2)\ |\ rf=fr, \ \forall r\in R\right\}.$$
Since  $E_V^k$ is an affine $k$-variety and $rf=fr$ for different $r \in R$ are algebraic equations, $E_V^R$ is a closed $k$-subvariety of $E_V^k$. Similarly, $G_V^R$ is a closed algebraic $k$-subgroup of $G_V^k$.

\subsection{Flags}\label{sec4.2.1}

Let $(\uik)=((i_1,k_1), \cdots, (i_m,k_m)) \in (I
\times \mathbb{N})^m$.
For such $(\uik)$, we shall write $(\ui, n\uk)=((i_1,nk_1), \cdots, (i_m,nk_m))$.

\subsubsection{} A  {\it $k$-flag} of type
$(\uik)$ in an $I$-graded $k$-vector space $V$ is a
sequence
$$\frak{f}=(V=V^0\supset V^1 \supset \cdots \supset V^m=0)$$
of $I$-graded vector spaces such that $V^{l-1}/V^l \simeq k^{\oplus
k_l}$ is concentrated at vertex $i_l$ for all $l=1,2, \cdots , m$.

Let $\mathcal{F}_{V, \uik}^k$ be the
$k$-variety of all $k$-flags of type $(\uik)$ in $V$,
and  $\widetilde{\mathcal{F}}_{V,
\uik}^k=\{(x,\frak{f}) \in E_V^k \times
\mathcal{F}_{V, \uik}^k \mid$ $\frak{f}$ is
$x$-stable$\}$, where $\frak{f}$ is $x$-stable if $x_h(V^l_{h'}) \subset
V^l_{h''}$, for all $h \in H$ and $ l=1, \cdots,m$.

The algebra group $G_V^k$ acts on $\mathcal{F}^k_{V,
\uik}$ by $g\cdot \frak{f} \mapsto g\frak{f}$, where
$$g\frak{f}=(gV^0\supset gV^1 \supset \cdots \supset gV^m=0)\ \ {\rm  if}\ \  \frak{f}=(V=V^0\supset V^1 \supset \cdots \supset V^m=0).$$
And $G_V^k$ acts diagonally on $\widetilde{\mathcal{F}}^k_{V,
\uik}$, i.e., $g\cdot (x,\frak{f})=
(gx,g\frak{f})$.

\subsubsection{}
An {\it $R$-flag} of type $(\uik)$ in an $I$-graded
$R$-module $V$ is a sequence
$$\frak{f}=(V=V^0\supset V^1 \supset \cdots \supset V^m=0)$$
of $I$-graded $R$-modules such that $V^{l-1}/V^l \simeq k^{\oplus
k_l}$ as $k$-vector spaces concentrated at vertex $i_l$ for all $l=1,2, \cdots , m$.

Similarly, let  $\mathcal{F}_{V, \uik}^R$ be the $k$-variety of all $R$-flags of type $(\uik)$ in $V$.
Note that if $V$ is an $R$-module, then a $k$-subspace $W \subset V$ is an $R$-submodule if and only if $(1+t)W=W$.
 So $\mathcal{F}_{V, \uik}^R$
 is a
closed subvariety of
$\mathcal{F}_{V, \uik}^k$.
Let $\widetilde{\mathcal{F}}_{V, \uik}^R=\widetilde{\mathcal{F}}_{V, \uik}^k \cap (E_V^k\times \mcal F_{V,\uik}^R)$.

\subsubsection{}
A {\it free $R$-flag} of type
$(\uik)$ in an $I$-graded free $R$-module $V$ is a
sequence
$$\frak{f}=(V=V^0\supset V^1 \supset \cdots \supset V^m=0)$$
of $I$-graded free $R$-modules such that $V^{l-1}/V^l \simeq
R^{\oplus k_l}$ is concentrated at vertex $i_l$ as $R$-modules for all $l=1,2, \cdots , m$.

Let $\mathcal{F}_{V, \uik}^{f} \subset
\mathcal{F}_{V, \ui,n\uk}^R$ be the subvariety
of all free $R$-flags of type $(\uik)$.
Let $\widetilde{\mathcal{F}}_{V, \uik}^{f}=\widetilde{\mathcal{F}}_{V, \ui,n\uk}^R \cap (E_V^R \times \mathcal{F}_{V,\uik}^{f})$ and $\mathcal{F}_{V, \uik}^{t}$ (resp. $\widetilde{\mathcal{F}}_{V,
\uik}^{ t}$) be the complement of $\mathcal{F}_{V, \uik}^{f}$ (resp. $\widetilde{\mathcal{F}}_{V, \uik}^{f}$)
in $\mathcal{F}_{V, \ui,n\uk}^R$ (resp. $ \widetilde{\mathcal{F}}_{V, \ui,n\uk}^R$).
The $G_V^R$-action on these $k$-varieties induced from its action on $\mcal F_{V, \ui, n\uk}^k$ (resp. $\wt{\mcal F}_{V, \ui,n\uk}^k$).

For any free $R$-module $V$, we write $V_0=V/tV$.
We  define the evaluation map as follows,
\begin{align*}
  e: \mathcal{F}_{V, \uik}^{f} &\rightarrow \mathcal{F}_{V_0, \uik}^k\\
    \frak{f}=(V\supset V^1 \supset \cdots \supset V^m=0) &\mapsto e(\frak{f})=(V_0\supset V^1_0 \supset \cdots \supset V^m_0=0).
\end{align*}

\begin{lem}[\cite{Mus}]\label{rem3.3}
  $\mathcal{F}_{V, \uik}^{f}$ is the $(n-1)$th-jet scheme over $\mathcal{F}_{V_0, \uik}^k$ and
  $\dim_k\mathcal{F}_{V, \uik}^{f}=n \dim_k \mathcal{F}_{V_0, \uik}^k$.
\end{lem}
{\bf Notations and convention}: (a)
From now on, we shall simply write
$\mcal F_{\uik}^f$ instead of $\mcal F_{V, \uik}^f$ unless we want to specify $V$.
Similarly, this convention is applied to other notations.

 (b) For any $(\uik)
 \in (I \times \mathbb{N})^m$ and each $i \in I$, let
$N_i(\uik)=\sum_{r<r'}k_rk_{r'}\delta_{ii_r} \delta_{ii_{r'}}$;
for each $h \in H$, let $N_h(\uik)
=\sum_{r'<r}k_{r'}k_r\delta_{h'i_{r'}}\delta_{h''i_{r}}$,
where $\delta$ is the Kronecker delta.

(c)  Dimension
always refers to $k$-dimension, so we shall denote it by ``${\rm dim}$"
instead of ``${\rm dim}_k$" and rank always refers to the rank of free
$R$-modules, and we shall therefore denote it by ``${\rm Rank}$".

\begin{prop}\label{prop3.1}
  \begin{itemize}
    \item [{\rm (a)}] $\mathcal{F}_{ \uik}^R$ is
    a projective variety.

    \item [{\rm (b)}] $\mathcal{F}_{\uik}^{f}$ is an open smooth subvariety of $\mathcal{F}_{\ui, n\uk}^{R}$
    and $\mathcal{F}_{ \uik}^{t}$ is a closed subvariety of $\mathcal{F}_{ \ui, n\uk}^{R}$.

    \item [{\rm (c)}] The evaluation map $e: \mathcal{F}_{V, \uik}^{f} \rightarrow \mathcal{F}_{V_0, \uik}^k$ is a vector bundle of dimension $(n-1)\sum_{i\in I}N_i(\uik)$. Hence the dimension of $\mathcal{F}_{
\uik}^{f}$ is $n\sum_{i\in I}N_i(\uik)$.
  \end{itemize}
\end{prop}

\begin{proof}
 {\rm (a)} It is well-known that $\mathcal{F}_{\uik}^k$ is a projective variety.
    Since $\mathcal{F}_{ \uik}^R$ is a closed subvariety of $\mathcal{F}_{\uik}^k$,
    it is also a projective variety.

{\rm (b)}
To show that $\mathcal{F}_{ \uik}^{f}$ is an open subset of $ \mathcal{F}_{\ui, n\uk}^{R}$,
we first consider an easy case that $I$ only contains one vertex.

Suppose that $V$ is a free $R$-module of Rank $l$, which can be thought as an $nl$-dim $k$-vector space.
  Let $Gr^k(sn,V)$ be
  the set of all $sn$-dimensional $k$-subspaces in $V$
  and $Gr^R(sn,V)=\{\frak{f} \in Gr^k(sn,V)\ |\ (1+t)\frak{f}=\frak{f}\}$.
  Let $Gr^{f}(s,V)$ be the set of all free
  $R$-submodules with Rank $s$ in $V$.
  Clearly, $Gr^{f}(s,V) \subset Gr^R(sn,V) \subset Gr^k(sn,V)$.

  Let
  $\widetilde{G}r^R(sn,V)=\left\{(W,b_W)\ |\ W \in Gr^R(sn,V)\right.$ and $
  b_W$ is a $k$-basis of $\left.W \right\}$.
  The first projection map $\pi:
  \widetilde{G}r^R(sn,V) \rightarrow Gr^R(sn,V)$ is a frame bundle.

  In general, for any free $R$-module $W \simeq R^{\oplus r}$,
    the $R$-module structure induces a nilpotent $k$-linear map $t: W
    \rightarrow W$ with $\dim({\rm Ker}(t))=r$.
We define
  \begin{align*}
\phi: \widetilde{G}r^R(sn,V) & \rightarrow {\rm Mat}(sn), \\
(W,b_W) & \mapsto M(b_W,t),
  \end{align*} where ${\rm Mat}(sn)$ is the set of all $sn \times sn$
  matrices and $M(b_W,t)$ is the matrix of $t$ under the basis
  $b_W$. Clearly, $\phi$ is a morphism of algebraic varieties.

For any $R$-submodule $W
    \subset V$ with $k$-dimension $ns$, $W$ is a free $R$-module if
    and only if $\dim({\rm Ker}(t|_W))=s$, i.e., $t|_W$ has maximal rank
    $(n-1)s$.
Therefore, $Gr^{f}(s,V)=\pi(\phi^{-1}({\rm Mat}(sn)_{rk=s(n-1)}))$, where
  ${\rm Mat}(sn)_{rk=s(n-1)}$ is the set of all matrices with rank $s(n-1)$.
  Since ${\rm Mat}(sn)_{rk=s(n-1)}$ is an open subset in ${\rm Mat}(sn)_{rk \leq s(n-1)}$,
\ $\phi^{-1}({\rm Mat}(sn)_{rk=s(n-1)})$ is open in $\phi^{-1}({\rm Mat}(sn)_{rk \leq
  s(n-1)})=\widetilde{G}r^R(sn,V)$.
  Moreover, $\pi$ is a principle $GL_{sn}(k)$-bundle and $\phi^{-1}({\rm Mat}(sn)_{rk=s(n-1)})$ is $GL_{sn}(k)$-stable,  $Gr^{f}(s,V)$ is open in
  $Gr^R(sn,V)$.

We now back to general cases.
By the above argument,
$\mcal F_l=\{\mathfrak f=(V^0\subset \cdots\subset V^m)\in \mcal F_{\ui, n\uk}^R|V^l\ \text{is a free}\ R-{\rm module}\}$ is
an open subset of $\mcal F_{\ui, n\uk}^R$.
    So $\mathcal{F}_{ \uik}^{f}=\cap_l \mcal F_l$ is an open subset of $\mcal F_{\ui, n\uk}^R$.
     The smoothness follows from Lemma \ref{rem3.3} and the notes after Lemma 1.2 in \cite{Mus}.
     This proves the first statement. The second statement follows from the first one.

{\rm (c)} Recall that $V_0=V/tV$.
It is well-known that $\mcal F_{V, \uik}^k=G_{V}^R/P^R$ (resp. $\mcal F_{V_0, \uik}^k=G_{V_0}^k/P^k$),
where $P^R$ (resp. $P^k$) is the parabolic subgroup of $G_V^R$ (resp. $G_{V_0}^k$) fixing a given flag in $\mcal F_{V, \uik}^k$ (resp. $\mcal F_{V_0, \uik}^k$).

Since $G_V^R=\mbf T\rtimes G_{V_0}^k $, where $\mbf T=\left\{Id+tB |B \in \End_R(V) \right\}$,
    any element $ A\in G_V^R$ can be written as $A=h\cdot A_0$ with $A_0
    \in G_{V_0}^k$ and $h \in \mbf T$.
    Then the evaluation map $e$ sends $h\cdot A_0$ to $A_0$. Therefore, $\forall x \in G_{V_0}^k$, we have $e^{-1}(x)=\mbf T/(\mbf T \cap
    P^R) \cdot x$. As a set, $\mbf T/(\mbf T \cap
    P^R) \cdot x$ is in 1-1 correspondence to $\mbf T/(\mbf T \cap
    P^R)$, and $\mbf T/(\mbf T \cap
    P^R)$ is a direct sum of quasi-lower triangular matrices with
    entries in $tR$ for all $i \in I$, which is clearly a $k$-vector space
    of
    dimension $(n-1)\sum_{i\in I}N_i(\uik)$.
It is clear that this gives a vector bundle structure.
\end{proof}

\begin{prop}\label{prop3.2}
  {\rm (a)}
    $\widetilde{\mathcal{F}}_{ \uik}^{f}$ is a
smooth irreducible variety, and the second projection map
$p_2: \widetilde{\mathcal{F}}_{ \uik}^{f}
\rightarrow \mathcal{F}_{ \uik}^{f}$ is a
vector bundle of dimension $n\sum_{h\in H}N_h(\uik)$. So the $k$-dimension
of $\widetilde{\mathcal{F}}_{ \uik}^{f}$
is
 $d(\uik):=n\sum_{i\in I}N_i(\uik) +n\sum_{h\in H}N_h(\uik)$.

{\rm (b)} Let $\pi_{ \uik}^f: \widetilde{\mathcal{F}}_{ \uik}^{f}
\rightarrow E_V^R$ be the first projection map. Then $(\pi_{ \uik}^f)_!{\bf 1}$ is
semisimple.
\end{prop}

\begin{proof}
{\rm (a)} Let $\widetilde{\mathcal{F}}_R^k=\left\{(x,\frak{f}) \in E_V^R \times s(\mathcal{F}_{V_0, \uik}^k) \ |\  \frak{f} \ {\rm is}\ x-{\rm stable}\right\}$,
where $s: \mathcal{F}_{V_0, \uik}^k \rightarrow \mathcal{F}_{V, \uik}^{f}$ is the zero section of the vector bundle $e$.
   By the same argument as that for Lemma 1.6 in \cite{L2},
   the second projection map $p_2': \widetilde{\mathcal{F}}_R^k
    \rightarrow  \mathcal{F}_{V_0, \uik}^k$ is a vector bundle.
    We  calculate the dimension of fibers.
    For any $\frak{f}=(V \supset V^1 \supset \cdots \supset V^m=0) \in \mathcal{F}^k_{V_0,\underline{i},\underline{k}}$,
   let $Z=(p_2')^{-1}(\frak f)$. The first projection identifies $Z$ with the set of all $x \in E_V^R$ such that $x_h(V_{h'}^l) \subset V_{h''}^l$ for all $h \in H$ and all $l=0,1,\cdots,m$. This is a linear subspace of $E_V^R$ because we can choose a basis for each $V_i$ such that $x_h$ are upper triangle matrices for each $h \in H$. Hence its dimension is equal to
   $$n \sum_{l' \leq l, h \in H}({\rm Rank}(V^{l'-1}_{h'})-{\rm Rank}(V^{l'}_{h'}))({\rm Rank}(V^{l'-1}_{h''})-{\rm Rank}(V^{l'}_{h''}))
   =n\sum_{h\in H}N_h(\uik).$$

We now consider the following cartesian square,
\begin{equation}\label{eq3.10}
 \xymatrix{
  \widetilde{\mathcal{F}}^{f}_{\uik} \ar[r]^{p_2} \ar[d]^b & \mathcal{F}^{f}_{ \uik}\ar[d]^e\\
  \widetilde{\mathcal{F}}_R^k \ar[r]^-{p_2'} & \mathcal{F}_{V_0, \uik}^k.}
\end{equation}
Since $p_2'$ is a vector bundle with rank $n\sum_{h\in H}N_h(\uik)$,
the second projection map $p_2$ is a vector bundle with rank $n\sum_{h\in H}N_h(\uik)$.
 The smoothness and irreducibility follow Proposition \ref{prop3.1}. This proves the first statement. The second statement follows from the first one.

{\rm (b)} Consider cartesian square (\ref{eq3.10}). By Proposition \ref{prop3.1}, the evaluation map $e$ is a vector bundle, then $b$ is also a vector bundle. Now consider the following commutative diagram
\begin{equation}\label{eq3.10.1}
 \xymatrix{ \widetilde{\mathcal{F}}^{f}_{\uik} \ar[r]^{b} \ar[dr]_{\pi_{ \uik}^f} &
 \widetilde{\mathcal{F}}_R^k \ar[d]^{p_1} \\ &  E_V^R}.
\end{equation}
Here  the first projection map $p_1$ is a proper map.
By the property in Section \ref{decomposition}(b), $(\pi_{ \uik}^f)_!{\bf 1}=p_{1!}b_!{\bf 1}$ is semisimple. Proposition follows.
\end{proof}

 Denote $\widetilde{L}^f_{V, \uik}=(\pi^f_{ \uik})_!{\bf 1} \in
\mathcal{D}(E_V)$, which is semisimple by Proposition \ref{prop3.2}.
Let
$$\textstyle L^f_{V, \uik}=\widetilde{L}^f_
{V, \uik}[d(\uik)+(n-1)\sum_{i\in I}N_i(\uik)],$$
where $d(\uik)$ is defined in Proposition \ref{prop3.2}(a).

Similarly, let $\pi^R_{\uik}:
\widetilde{\mathcal{F}}^{R}_{\uik}
\rightarrow E_V^R$ be the first projection map. We define
$\widetilde{L}^R_{V, \uik}=
(\pi^R_{\uik})_!{\bf 1}$. 

 Let $\mathcal{P}^f_V$ (resp. $\mathcal{P}^R_V$) be the full subcategory of
$\mathcal{M}(E_V^R)$ consisting of  direct
sums of the simple perverse sheaves $L$ which are the direct summands
of $\widetilde{L}^f_{V, \uik}$ (resp.
$\widetilde{L}^R_{V, \uik}$) up to shift for
some $(\uik) \in (I \times \mathbb{N})^m$.
Let $\mathcal{Q}^f_V$ (resp.
$\mathcal{Q}^R_V$) be the full subcategory of $\mathcal{D}(E_V^R)$
whose objects are isomorphic to finite direct sums of $L[d]$ for
various simple perverse sheaves $L$ in $\mathcal{P}^f_V$ (resp.
 $\mathcal{P}^R_V$) and various $d \in
\mathbb{Z}$.

\subsection{Restriction functor}\label{res}
By abuse of notations, from now on, we write $E_V$ (resp. $G_V$) instead of $E_V^R$ (resp. $G_V^R$) unless we specify.
Let $W$ be an $I$-graded free
$R$-submodule of $V$ such that $T=V/W$ is also a free $R$-module.
Let $F=\{x\in E_V|x(W)\subset W\}$. Denote $E_{T,W}=E_T\times E_W$ and $G_{T,W}=G_T\times G_W$.
Consider the following diagram
\begin{equation}\label{eqres}
\xymatrix{ E_{T, W} & F \ar[l]_-{\kappa} \ar[r]^{\iota} & E_V,}
\end{equation}
where
$\iota$ is an embedding and $\kappa(x)=(x_W,x_T)$. Here $x_W = x|_W$ and $x_T$ is the induced map $\overline{x} : V/W \rightarrow  V/W$. For any $B\in \mathcal{D}(E_V)$, we define
$$\overline{\Res}^V_{T,W}B=\kappa_! \iota ^*B.$$
 We note that, in general,
$\overline{\Res}^V_{T,W}B \not \in \mathcal{Q}^f_{T,W}$, even for $ B \in \mathcal{Q}^f_V$,
where $\mcal Q^f_{T,W}$ is defined similarly as $\mcal Q_V^f$ replacing $E_V$ by $E_{T,W}$.
In fact, given a free $R$-flag
$\frak{f}=(V^0 \supset V^1 \supset \cdots \supset V^m=0)$, and $W \subset
V,\  T=V/W$ being free $R$-modules, the induced flags
\begin{equation}\label{eq3.11.1}
\frak{f}_T:=((V^0+W)/W
\supset (V^1+W)/W \supset \cdots \supset (V^m+W)/W=0)
\end{equation}
\begin{equation}\label{eq3.11.2}
 \frak{f}_W:=(V^0
\cap W \supset V^1 \cap W \cdots \supset V^m \cap W=0)
\end{equation}
are
no longer free $R$-flags, since $V^l \cap W$ and $(V^l +W)/W$ are no longer free $R$-modules in general.

\begin{lem}\label{lem3.3}
  $\overline{\Res}^V_{T,W}(B)$ is semisimple in $\mathcal{D}^b_{G_{T, W}}(E_{T,W})$ for all $B \in \mathcal{Q}^f_V$.
\end{lem}
\begin{proof}
It is sufficient to prove that $\kappa_!\iota^*(B)$ is
semisimple for a simple perverse sheaf $B\in \mcal Q_V^f$.
Let $\iota': E_{T,W} \hookrightarrow F$ be the natural embedding map.
Then $(\iota')^!\iota^*: \mcal D_{G_V}(E_V)\rightarrow \mcal D_{G_{T,W}}(E_{T,W})$
is the hyperbolic localization functor defined in \cite{Braden}.
By (1) in \cite{Braden}, we have
\[\overline{\Res}^V_{T,W}(B)=\kappa_!\iota^*(B)\simeq (\iota')^!\iota^*(B).\]
By Theorem 8 in \cite{Braden}, $\overline{\Res}^V_{T,W}$ preserves purity.
Since $B$ is pure, $\overline{\Res}^V_{T,W}(B)$ is also pure.
Therefore, by the property in Section \ref{decomposition}(a), $\overline{\Res}^V_{T,W}(B)$ is semisimple.
\end{proof}

Since the objects in $\mathcal{Q}^R_V$ are semisimple complexes, every object $A \in \mathcal{Q}^R_V$ can be uniquely written into $A=A^f \oplus A^t$ such that $A^f \in \mathcal{Q}^f_V$, $A^t \in \mathcal{Q}^R_V \setminus \mathcal{Q}^f_V$ and $A^f$ is the maximal subobject of $A$ which is in $ \mathcal{Q}^f_V$. Therefore we can define a projection functor
$P_f: \mathcal{Q}^R_V \rightarrow \mathcal{Q}^f_V$ sending $A$ to $A^f$. 

\begin{defn}
 $\widetilde{\Res}^V_{T,W}(B):=P_f(\overline{\Res}^V_{T,W}(B))$.
\end{defn}

\begin{prop}\label{prop3.3}
  $\widetilde{\Res}^V_{T,W}(B) \in \mathcal{Q}^f_{T,W}$ if $B \in
  \mathcal{Q}_V^f$.
\end{prop}
\begin{proof}
This follows directly from the definition of
$\widetilde{\Res}^V_{T,W}$.
\end{proof}

\subsection{Induction functor}\label{ind}

Recall that $W$ is a free $R$-submodule of $V$ such that $T=V/W$ is also a free $R$-module.
 Let $P$
be the stabilizer of $W$ in $G_V$ and $U$ be the unipotent radical of
$P$.
 Consider the following diagram:
 \begin{equation}\label{eq3.12.1}
\xymatrix{ E_{T, W} & G_V \times^U F \ar[l]_-{p_1}\ar[r]^-{p_2}&
G_V \times^P F \ar[r]^-{p_3} &E_V.}
 \end{equation}
Here $p_1(g,x)=\kappa(x),\
p_2(g,x)=(g,x)$, and $p_3(g,x)=g(\iota(x))$, where $\kappa$ and $\iota$
are the maps in Diagram (\ref{eqres}). For any $A \in
\mathcal{D}_{G_{T, W}}(E_{T, W})$, we define
$$\widetilde{\Ind}^V_{T,W}A=p_{3!}p_{2\flat}p^*_1A.$$
Here $p_{2\flat}$ is
well defined since $p_2$ is a principle $G_{T, W}$-bundle.

\begin{prop}\label{prop3.5} $\widetilde{\Ind}^V_{T,W}A \in \mathcal{Q}^f_V$ if $A \in
\mathcal{Q}^f_{T,W}$.
\end{prop}
\begin{proof}
The proof is the same as that for Lemma 9.3.2 in \cite{L5}.
For later use, we present here.
Since  $\widetilde{\Ind}^V_{T,W}$ is additive, it is enough to prove the proposition for $A=\widetilde{L}_{T, \ui', \uk'}^f \boxtimes \widetilde{L}^f_{W, \ui'', \uk''}$, where
$(\ui', \uk')=((i'_1,k'_1)\cdots (i'_{m},k'_m))$ and
$(\ui'', \uk'')=((i''_1,k''_1)\cdots (i''_{l},k''_l))$.
For such $(\ui', \uk')$ and $(\ui'', \uk'')$, we denote
\begin{equation}\label{eqik}
(\ui'\ui'', \uk' \uk'')=( (i'_1,k'_1)\cdots (i'_{m},k'_m),(i''_1,k''_1)\cdots (i''_{l},k''_l)).
\end{equation}
We simple write $(\uik)= (\ui'\ui'', \uk' \uk'')$.
 Let
$$\mathcal{F}^0_{\uik}=\left\{ (V^0 \supset
V^1 \supset \cdots \supset V^m \cdots \supset V^{m+s}=0) \in
\mathcal{F}_{\uik}^f\ |\ V^m=W \right\}$$
 and
$\widetilde{\mathcal{F}}^0_{\uik}=
\widetilde{\mathcal{F}}_{\uik}^f \bigcap (F
\times \mathcal{F}^0_{\uik})$.
Now consider the following diagram,
\begin{equation}\label{eq3.13}
\xymatrix{
\widetilde{\mathcal{F}}^f_{T,\ui', \uk'}\times
\widetilde{\mathcal{F}}^f_{W,\ui'', \uk''} \ar[d]_-{\pi^f_{T,W}} \ar @{} [dr] |{\fbox{1}}& G_V
\times^U\widetilde{\mathcal{F}}^0_{\uik} \ar[l]_-{\widetilde{p}_1} \ar[d]_-{u'} \ar[r]^-{\widetilde{p}_2} \ar @{} [dr] |{\fbox{2}}&
 G_V
\times^P\widetilde{\mathcal{F}}^0_{\uik} \ar[r]^-i \ar[d]^-u &
\widetilde{\mathcal{F}}^f_{\uik} \ar[d]^{\pi^f_{\uik}}
\\
E_{T, W} &
 G_V \times^UF \ar[l]_-{p_1} \ar[r]^-{p_2} & G_V \times^PF
\ar[r]^-{p_3} & E_V.
}
\end{equation}

Here the vertical maps are all projection maps and $i$ is an
identity map. The squares $\fbox{1}$ and $\fbox{2}$ are both
cartesian squares and $\widetilde{p}_2$ is a principle  $G_{T,W}$-bundle. It follows that
$$p^*_1
(\pi^f_{T,W})_!{\bf 1}=u'_!\widetilde{p}_1^*{\bf 1}=u'_!\widetilde{p}_2^*{\bf 1}=p^*_2
u_!{\bf 1}.$$ So
$$p_{3!}p_{2\flat}p^*_1A=p_{3!}p_{2\flat}p^*_1
(\pi^f_{T,W})_!{\bf 1}=p_{3!}u_!{\bf 1}=(\pi^f_{\uik})_!{\bf 1} \in \mathcal{Q}^f_V.$$
\end{proof}
\begin{rem}\label{rem3.6}
From the above proof, we have $\widetilde{\Ind}^V_{T,W}
(\widetilde{L}^f_{T, \ui', \uk'} \boxtimes
\widetilde{L}^f_{W, \ui'', \uk''})=\widetilde{L}^f_{V, \ui'\ui'', \uk'\uk''}$.
\end{rem}

\begin{lem}\label{lem3.3.1}
 Let $b: Y \rightarrow X$ be a fiber bundle with $d$ dimensional connected smooth irreducible fiber. 
  If $B=b^* A$ for some $A \in \mathcal{D}^b(X)$, then
  $\mathbb{D}b_!B=(b_!\mathbb{D}B)[2d]$.
\end{lem}
\begin{proof}
 Since $b_!B=b_!b^*A=A[-2d]$, we have
  $$\mathbb{D}b_!B=\mathbb{D}(A[-2d])=(\mathbb{D}A)[2d],$$ and
  $$b_!\mathbb{D}B=b_!\mathbb{D}b^*A=b_!b^!(\mathbb{D}A)=b_!b^*(\mathbb{D}A)[2d]=\mathbb{D}A.$$
  Lemma follows.
\end{proof}
Denote by $d_1$ (resp. $d_2$) the dimension of the fibers of $p_1$
(resp. $p_2$), where $p_1$ and $p_2$ are the maps defined in Diagram (\ref{eq3.12.1}).
After simple calculations, we have $ d_2={\rm dim}P/U$ and $d_1={\rm
dim}G_V/U+n\sum_{h \in H} {\rm Rank}(T_{h'}){\rm Rank}(W_{h''})$.

\begin{prop}\label{prop3.5.1}
Let $A$ be a direct summand of $\widetilde{L}^f_{T, \ui', \uk'} \boxtimes \widetilde{L}^f_{W, \ui'', \uk''}$, then
$$\mathbb{D}(\widetilde{\Ind}^V_{T,W}(A))=\widetilde{\Ind}^V_{T,W}(\mathbb{D}(A))[2d_1-2d_2+2(n-1)\sum_{i\in I}{\rm Rank}(T_i){\rm Rank}(W_i)].$$
\end{prop}

\begin{proof}
  Since $\mathbb{D}$ is additive, it is enough to consider $A=\widetilde{L}^f_{T, \ui', \uk'} \boxtimes \widetilde{L}^f_{W, \ui'', \uk''}$.
From the proof of Proposition \ref{prop3.5}, we have
$$ \mathbb{D}(\widetilde{\Ind}^V_{T,W}(A))=\mathbb{D}(p_{3!}p_{2\flat}p_1^*(\pi^f_{T,W})_!{\bf 1})
  =\mathbb{D}((\pi^f_{\uik})_!i_!\widetilde{p}_{2\flat}\widetilde{p}_1^*{\bf 1}).$$

From the proof of Proposition \ref{prop3.2}, $\pi^f_{\uik}=p\circ b$ such that $p$ is a proper map and $b$ is a vector bundle with rank $(n-1)\sum_{i\in I}N_i(\uik)$. By Lemma \ref{lem3.3.1},
\begin{eqnarray*}
&& \mathbb{D}((\pi^f_{\uik})_!i_!\widetilde{p}_{2\flat}\widetilde{p}_1^*{\bf 1})
=p_!b_!i_!\mathbb{D}\widetilde{p}_{2\flat}\widetilde{p}_1^*{\bf 1}[2(n-1)\sum_{i\in I}N_i(\uik)]\\
&=&(\pi^f_{\uik})_!i_!\widetilde{p}_{2\flat}\widetilde{p}_1^*(\mathbb{D}{\bf 1})[2d_1-2d_2+2(n-1)\sum_{i\in I}N_i(\uik)].
\end{eqnarray*}
On the other hand, by a similar reason, we have
\begin{eqnarray*}
\widetilde{\Ind}^V_{T,W}(\mathbb{D} A)&=&p_{3!}p_{2\flat}p_1^*(\mathbb{D}(\pi^f_{T,W})_!{\bf 1})\\
&=&p_{3!}p_{2\flat}p_1^*p'_!b'_!(\mathbb{D}{\bf 1})[2(n-1)\sum_{i\in I}(N_i(\ui', \uk')+N_i(\ui'', \uk''))]\\
&=&p_{3!}p_{2\flat}p_1^*(\pi^f_{T,W})_!(\mathbb{D}{\bf 1})[2(n-1)\sum_{i\in I}(N_i(\ui', \uk')+N_i(\ui'', \uk''))]\\
&=&(\pi^f_{\uik})_!i_!\widetilde{p}_{2\flat}\widetilde{p}_1^*(\mathbb{D}{\bf 1})[2(n-1)\sum_{i\in I}(N_i(\ui', \uk')+N_i(\ui'', \uk''))].
\end{eqnarray*}
Here we use a similar decomposition $\pi^f_{T,W}=p'\circ b'$
as we did for $\pi_{\uik}^f$ such that $p'$ is a proper map and $b'$ is a vector bundle with rank $(n-1)\sum_{i\in I}(N_i(\ui', \uk')+N_i(\ui'', \uk''))$.

By the definition of $N_i(\uik)$,
it is straightforward to check that
\begin{equation}\label{eq3.13.1}
N_i(\ui'\ui'', \uk'\uk'')-N_i(\ui', \uk')-N_i(\ui'', \uk'')=\sum_{r,r'} k'_{r}k''_{r'}\delta_{ii_r}\delta_{ij_{r'}}\\
={\rm Rank}(T_i){\rm Rank}(W_i).
\end{equation}
Proposition follows.
\end{proof}

Let
\begin{equation}\label{eq3.14}
  \Ind^V_{T,W}A=\widetilde{\Ind}^V_{T,W}A[d_1-d_2+(n-1)\sum_{i\in I}{\rm Rank}(T_i){\rm Rank}(W_i)],\
\end{equation}
\begin{equation}\label{eq3.15}
  \Res^V_{T,W}A=\overline{\Res}^V_{T,W}A[d_1-d_2-2{\rm dim}G_V/P+(n-1)\sum_{i\in I}{\rm Rank}(T_i){\rm Rank}(W_i)].
\end{equation}
Then we have the following corollary.

\begin{cor}\label{cor3.0.1}
  $\mathbb{D}(\Ind^V_{T,W}(A))=\Ind^V_{T,W}(\mathbb{D}(A)).$
\end{cor}

\begin{proof}
This follows directly from Proposition \ref{prop3.5.1}.
\end{proof}

\begin{cor}\label{cor3.1}
$\Ind^V_{T,W}(L^f_{T,\ui', \uk'} \boxtimes
L^f_{W,\ui'', \uk''})=L^f_{V, \ui'\ui'', \uk' \uk''}$,
where $(\ui'\ui'', \uk' \uk'')$ is defined in (\ref{eqik}).
\end{cor}
\begin{proof}
Corollary follows from (\ref{eq3.14}) and Remark \ref{rem3.6}.
\end{proof}

\subsection{Bilinear form}\label{section4.4}

Recall that, for any $G$-equivariant semisimple complexes $A,B$ on algebraic variety $X$,
Lusztig defines an integer number
$d_j(X, G; A, B)$ for any $j\in \mbb Z$ in \cite[Section 1]{L9}.
For later use, we list some properties as follows.
We refer to \cite{ L5} for more details.
\begin{itemize} \item[(a)]$d_j(X,G;A,B)=d_j(X,G;B,A)$.

\item[(b)] $d_j(X,G;A[n],B[m])=d_{j+n+m}(X,G;A,B)$ for any $m,n \in \mathbb{Z}$.

\item[(c)] $d_j(X,G;A\oplus A',B)=d_j(X,G;A,B)+d_j(X,G;A',B)$.

\item[(d)] If $A$ and $B$ are perverse sheaves, then
$d_j(X,G;A,B)=0$ for all $j>0$. If, in addition, $A$ and $B$ are
simple and $B \simeq \mathbb{D}A$
, then
$d_0(X,G;A,B)$ is $1$ and is zero otherwise.
\end{itemize}

\begin{lem}[\cite{GL}]\label{lem3.4} Let $A \in \mathcal{Q}^f_{T,W}$ and $B \in \mathcal{Q}^f_V$. Then for any $j \in \mathbb{Z},$
$$d_j(E_{T, W},G_{T, W};A,\overline{\Res}^V_{T,W}B)=d_{j'}(E_V,G_V;\widetilde{\Ind}^V_{T,W}A,B),$$
where $j'=j+2{\rm dim}G_V/P$.
\end{lem}
\begin{prop}\label{prop3.6}
  Let $A \in \mathcal{Q}^f_{T,W}$ and $B \in \mathcal{Q}^f_V$. Then for any $j \in \mathbb{Z},$
$$d_j(E_{T, W},G_{T, W}; A,\Res^V_{T,W}B)=d_{j}(E_V,G_V;\Ind^V_{T,W}A,B).$$
\end{prop}
\begin{proof} This follows directly from definitions (\ref{eq3.14}), (\ref{eq3.15}) and Lemma
\ref{lem3.4}.
\end{proof}

\subsection{Fourier-Deligne transform}

Let us now
consider a new orientation of the given quiver. Denote the source of
the arrow $h$ by ${}'\!h$ and its target by ${}''\!h$ for the
new orientation. Recall that we denote the source of
the arrow $h$ by $h'$ and its target by $h''$ for the
old orientation. Let $H_1=
\left\{ h \in H \mid {}'\!h=h', {}''\!h=h'' \right\}$ and $H_2= \left\{
h \in H \mid {}'\!h=h'',{}''\!h=h' \right\}$. For a given $I$-graded
free $R$-module $V$, denote
$$E_V=\oplus_{h \in H_1} \Hom_R(V_{h'},V_{h''}) \oplus \oplus_{h \in
H_2} \Hom_R(V_{h'},V_{h''}),$$
$${}'\!E_V=\oplus_{h \in H_1} \Hom_R(V_{h'},V_{h''}) \oplus \oplus_{h \in
H_2} \Hom_R(V_{h''},V_{h'}),$$
$$\dot{E}_V=\oplus_{h \in H_1} \Hom_R(V_{h'},V_{h''}) \oplus
\oplus_{h \in H_2} \Hom_R(V_{h'},V_{h''}) \oplus \oplus_{h \in H_2}
\Hom_R(V_{h''},V_{h'}).$$ Then we have the natural projection maps
\begin{equation}\label{eqFD}
  \xymatrix{E_V & \dot{E}_V \ar[l]_-{s}\ar[r]^-{t}& {}'\!E_V.}
\end{equation}

Consider $\Hom_R(V_{h'}, V_{h''})$ as a subset of $\Hom_k(V_{h'}, V_{h''})$, then we define a
map $\mathcal{T}_V: \dot{E}_V  \rightarrow k $ by
\begin{equation}\label{eq3.16.1}
\mathcal{T}_V(a,b,c)=\sum_{h \in H_2}  tr(V_{h'}
\xrightarrow{b} V_{h''} \xrightarrow{c} V_{h'}),
\end{equation}
 where $tr$ is the
trace function of the endomorphism of $k$-vector spaces. Clearly, $\mathcal{T}_V$ is
a bilinear map.

Define
\begin{align*}
  \Phi: \mathcal{D}(E_V) &\rightarrow
\mathcal{D}({}'\!E_V)\\
 A &\mapsto t_!(s^*(A) \otimes
\mcal L_{\mathcal{T}_V})[d_V],
\end{align*} where $d_V={\rm dim} (\oplus_{h \in H_2}
\Hom_R(V_{h'}, V_{h''}))=n\sum_{h \in H_2}{\rm Rank}(V_{h'}){\rm Rank}(V_{h''})$
and $\mcal L_{\mathcal{T}_V}$ is a rank 1 local system on $\dot E_V$ defined in ~\cite[Section 8.1.11]{L5}.

Similarly, we have the projection maps
\begin{equation}\label{eqFD2}
\xymatrix{E_T \times E_W & \dot{E}_T \times
\dot{E}_W \ar[l]_-{\overline{s}}\ar[r]^-{\overline{t}}& {}'\!E_T \times {}'\!E_W.}
\end{equation}
 Define $\overline{\mathcal{T}}: \dot{E}_T \times \dot{E}_W  \rightarrow k$
by $\overline{\mathcal{T}}:=\mathcal{T}_T+\mathcal{T}_W$,
where
$\mathcal{T}_T: \dot{E}_T \rightarrow k$ (resp. $\mathcal{T}_W:
\dot{E}_W \rightarrow k$) is defined in (\ref{eq3.16.1}) replacing $V$ by $T$ (resp. $W$). In a similar
fashion, one can define
\begin{align*}
\Phi: \mathcal{D}(E_T \times E_W)
& \rightarrow \mathcal{D}({}'\!E_T \times {}'\!E_W)\\
A &\mapsto \overline{t}_!(\overline{s}^*(A) \otimes
\mcal L_{\overline{\mathcal{T}}})[d_T+d_W].
\end{align*}

\begin{prop}\label{prop3.7}
For any $B \in \mathcal{Q}^f_V$, we have $$\Phi
\overline{\Res}^V_{T,W}(B)=\overline{\Res}^V_{T,W} \Phi (B)[\pi],$$ where
$\pi=n\Sigma_{h \in H_2}({\rm Rank}(T_{h''}){\rm
Rank}(W_{h'})-{\rm Rank}(T_{h'}){\rm Rank}(W_{h''})).$
\end{prop}
\begin{proof}
  The proof is the same as that for Proposition 10.1.2 in \cite{L5},
  except the dimension argument.
  By tracking the proof in \cite{L5}, we have
  $$\pi=d_T+d_W-d_V-2m.$$
   Here
  $m$ is the dimension of the vector bundle $\dot q: \dot F\rightarrow \Psi$, where
  $\dot F=\{x\in \dot E_V|s(x)\in F, t(x)\in {}'\!F\}$, ${}'\!F=\{x\in  {}'\!E_V \mid x_h(W_{{}'\!h}) \subset W_{{}''\!h},\
\forall h \in H \}$ and $\Psi$ is the fiber product of maps $\kappa$ in (\ref{eqres}) and $\overline s$ in (\ref{eqFD2}).
For any $(x,y,z)\in \dot F$, $\dot{q}(x,y,z)= ((x,y),(x_W,y_W,z_W),(x_T,y_T,z_T))$, where $x_W$ and $x_T$ are defined in Section \ref{res}.
A direct calculation shows that $$\textstyle m=n\sum_{h \in H_2}{\rm Rank}(T_{h'}){\rm Rank}(W_{h''}).$$
Therefore,
$ \pi=n\sum_{h \in H_2}({\rm Rank}(T_{h''}){\rm
Rank}(W_{h'})-{\rm Rank}(T_{h'}){\rm Rank}(W_{h''})).$
\end{proof}

\begin{lem}\label{lem3.5}
$\Phi(\widetilde{L}_{V,\uik}^f)=\
{}'\!\widetilde{L}_{V,\uik}^f[M]$
for some $M$.
\end{lem}
\begin{proof}
The proof is the same as that for Proposition 10.2.2 in \cite{L5}.
\end{proof}

\begin{cor}\label{cor3.2}$\Phi (\widetilde{\Res}^V_{T,W}(B))=\widetilde{\Res}^V_{T,W}( \Phi
(B))[\pi]$.
\end{cor}
\begin{proof} From Lemma \ref{lem3.5},
 $\Phi(\mathcal{Q}_V^f) \subset {}'\!\mathcal{Q}_V^f$, where ${}'\!\mathcal{Q}_V^f$ is defined similarly as $\mathcal{Q}_V^f$ for ${}'\!E_V$.
 By the same argument, $\Phi({}'\!\mathcal{Q}_V^f) \subset \mathcal{Q}_V^f$.
Since $\Phi(\Phi(K))=K$ (see 10.2.3 in
 \cite{L5}), for any $ K \in \mathcal{Q}_V^R \setminus \mathcal{Q}_V^{f}$,
  if $\Phi(K) \in{}'\!\mathcal{Q}_V^{f}$, then $K=\Phi(\Phi(K)) \in \mathcal{Q}_V^{f}$. This is a contradiction. Therefore, $\Phi(K) \not \in {}'\!\mathcal{Q}_V^{f} $ for any $K \in \mathcal{Q}_V^R \setminus \mathcal{Q}_V^{f}$.

By definition of $\widetilde{\Res}^V_{T,W}$, the corollary follows
from Proposition \ref{prop3.7}.
\end{proof}

\begin{cor}\label{cor3.4}
{\rm (a)}  $\Phi (\Res^V_{T,W}(B))=\Res^V_{T,W}( \Phi (B))$.\\
{\rm (b)} $\Phi (\Ind^V_{T,W}(B))=\Ind^V_{T,W}( \Phi (B))$
\end{cor}
\begin{proof}
Part (a) follows from  Corollary \ref{cor3.2} and the definition of $\Res^V_{T,W}$ in (\ref{eq3.15}).
The proof of Part (b) is the same as that for Proposition 10.2.6 in \cite{L5}.
\end{proof}

\subsection{Additive generators}

Let $v$ be an indeterminate and $\mathbf{A}=\mathbb{Z}[v,v^{-1}]$.
Let $\mathbf{M}_V$ be the Grothendieck group of the category which
consists of all direct sums of $L_{V,\uik}$
for various $(\uik)$ and their shifts. Define an $\mathbf A$-action on $\mathbf M_V$ by
\begin{equation}\label{action}
  v^n \cdot  L=L[n].
\end{equation}
 Then $\mathbf M_V$ is
an  $\mathbf A$-module generated by
$L_{V,\uik}$. Let $\mathbf K_V$ be the
Grothendieck group of category $\mathcal{Q}^f_V$.
Then under the action (\ref{action}), $\mathbf K_V$ is an
$\mathbf A$-module generated by the simple perverse sheaves in $\mathcal{P}^f_V$.

\begin{prop}\label{thm3.2}
$\mathbf M_V \simeq \mathbf K_V$ as an $\mathbf A$-module, i.e.,
$\left\{L_{V,\uik}\ |\ \forall (\uik)\right\}$  are the additive generators of $\mathbf K_V$.
\end{prop}
The proof of the this proposition is quite involved.
To reduce the length of the paper, we skip the proof.
The detailed proof can be found in \cite[Section 4.2.6]{Fan0}.

\subsection{The algebra ($\mathbf K$, $\Ind$)}
For an $I$-graded free $R$-module $V$, we denote
 $|V|=({\rm Rank}(V_i))_{i \in I}$.
 It is important to notice that, given two different $I$-graded free $R$-modules $V$ and $V'$ with $|V|=|V'|$, $\mathbf K_V \simeq \mathbf K_{V'}$ since $E_V$ and $E_{V'}$ are isomorphism spaces.
We shall write $\mathbf K_{|V|}$ (resp. $\mcal P_{|V|}^f$) instead of $\mathbf K_V$ (resp. $\mcal P_V^f$).
Moreover, the functors $\Ind^V_{T,W}$ and $
\Res^V_{T,W}$ can be rewritten as $\Ind^{|T|+|W|}_{|T|,|W|}$ and $
\Res^{|T|+|W|}_{|T|,|W|}$, respectively.
Let $\mathbf K=\oplus_{\nu \in \mathbb{N}I} \mathbf K_{\nu}$.
We define a multiplication on $\mbf K$ by
 \begin{align*}
\Ind: \mathbf K
\times \mathbf K &\rightarrow \mathbf K\\ 
 (A,B) &\mapsto
\Ind^{\tau+\omega}_{\tau,\omega}(A\otimes B)
 \end{align*}
for homogenous elements $A,B$ with $A \in \mathbf K_{\tau}$ and
$B \in \mathbf K_{\omega}$.
\begin{thm}\label{thm3.3}
{\rm (a)} $(\mathbf K, \Ind)$ is an $\mbb N[I]$-graded associative $\mathbf A$-algebra, which is called the geometric composition algebra.

{\rm (b)} All simple perverse sheaves in $\mathcal{P}_{\nu} ^f$ for various $\nu\in \mbb N[I]$ form an
$\mathbf A$-basis of $\mathbf K$.
 This basis is called the
canonical basis.

\end{thm}
\begin{proof} (a) follows from Proposition \ref{thm3.2}, Corollary \ref{cor3.1} and additivity of $\Ind$.
A proof of associativity without help of Proposition \ref{thm3.2} can be found in \cite[Section 3.3]{L2}.
(b) follows from the definition of $\mathbf K$.
\end{proof}

\section {Ringel Hall algebra over exact categories}\label{sec4.1}

To give an algebraic construction corresponding to the algebra $\mbf K$,
one needs to consider projective representations of quivers over $R$.
The category of projective representations of quivers over $R$ is an exact category.
Hubery defines a Hall algebra over an exact category in \cite{H1}.
It is natural to ask if there exists a coalgebra structure of this algebra.
 In general, the coalgebra structure can not be obtained by any twist analogue to that for abelian categories
 even when the homological dimension of the
 category is 1.

\subsection{Exact categories}
Let $\mathcal{A}$ be an additive
category which is a full subcategory of an abelian category
$\mathcal{B}$ and closed under extension in $\mathcal{B}$. Let
$\mathcal{E}$ be a class of sequences
$$\xymatrix{0 \ar[r] & M' \ar[r]^{i}& M \ar[r]^{j} & M'' \ar[r] & 0}$$
 in $\mathcal{A}$ which are exact in the abelian category
$\mathcal{B}$. A map $f$ is called an inflation (resp. a deflation )
if it occurs as the map $i$ (resp. $j$ ) of some members in
$\mathcal{E}$. Inflations and deflations will be denoted by
$M'\rightarrowtail M$ and $M \twoheadrightarrow M''$, respectively.
The pair $M'\rightarrowtail M \twoheadrightarrow M''$ is called a
conflation.
An {\it exact category} is the additive category
$\mathcal{A}$ equipped with a family $\mathcal{E}$ of the short
exact sequences of $\mathcal{A}$ satisfying certain properties.
We refer to \cite{Q} for more details.

\subsection{Representations of quivers over commutative rings}\label{sec4.2}
 A {\it representation} $(V,x)$
 of a quiver $\Gamma =(I,H,s,t)$ over a commutative ring $R$ is an $I$-graded  $R$-module $V$ together with a set
 $\left\{x_{h}\right\}_{h\in H}$ of
 $R$-linear transformations
 $x_{h}:V_{h'}\rightarrow V_{h''}$.

 A {\it homomorphism} from
 one representation $(V,x)$ to another representation $(W,y)$ is a collection
 $\left\{g_{i}\right\}_{i\in I}$ of $R$-linear maps
 $g_{i}:V_{i}\rightarrow W_{i}$, such that
 $g_{h''}x_{h}=y_{h}g_{h'}$ for all
  $h \in H$. If all $g_{i}$ are $R$-isomorphisms, $(V,x)$ and $(W,y)$ are
 said to be {\it isomorphic}.

Let $Rep_R(\Gamma)$ be the category of representations of $\Gamma$
over $R$, which is an abelian category. If $V$ is an $I$-graded projective $R$-module,
 then  $(V,x)$ is called
a {\it locally projective representation} of $\Gamma$ over $R$.
All such representations form a full subcategory of
 $Rep_R(\Gamma)$, denoted by $Rep_R^f(\Gamma)$.
In this case that $V$ is a locally projective representation of $\Gamma$ over $R$, the dimension
 vector $|V|$ is well-defined.
\begin{lem}\label{lem3.2}$Rep_R^f(\Gamma)$ is an exact category with
homological dimension 1.
\end{lem}
\begin{proof}It is easy to see $Rep_R^f(\Gamma)$ is an additive
category. Let $\mathcal{E}$ be the set of all possible short exact
sequences in $Rep_R^f(\Gamma)$. Then $Rep_R^f(\Gamma)$ with the class $\mathcal{E}$ is an
exact category. In this case, an inflation is an injective map,
such that the cokernel is an $I$-graded projective $R$-module and a
deflation map is a surjective map, such that the kernel is an
$I$-graded peojective $R$-module.

Let $A=R\Gamma$. To show the homological dimension of $Rep_R^f(\Gamma)$ is 1, it is enough to show 
that sequence
 \begin{equation}\label{resolution}
\xymatrix{ 0 \ar[r] & \oplus_{\rho \in H}Ae_{\rho''} \otimes_R
e_{\rho'}X \ar[r]^{f} & \oplus_{i \in I}Ae_i \otimes_R
e_iX \ar[r]^-g & X \ar[r] & 0}
\end{equation}
is exact for any locally projective left $A$-module $X$,
where $e_i$ is the trivial path
for the vertex $i$, and $g(a \otimes x)=ax, f(a\otimes x)=a\rho \otimes x-a \otimes \rho x$.
The proof of exactness is the same as that for the stand resolution in \cite{craw1992}.
\end{proof}

Here the homological dimension 1 refers to $\Ext^n(X,Y)$ vanishing
for all $n \geq 2\ {\rm and}\ X,Y \in \mathcal{A}$. We refer to
\cite[Chapter 6]{FS} for the definition of $\Ext^n(X,Y)$ in an exact category.

\subsection{The Hall algebra over an exact category}

 Let $\mathcal{A}$ be a finitary and small
exact category. Denote by $W^L_{XY}$ the set of all conflations $Y
\rightarrowtail L \twoheadrightarrow X$. The group $Aut(X) \times
Aut(Y)$ acts on $W^L_{XY}$ via:
$$\xymatrix{ Y \ar[r]^{f} \ar[d]^{\eta} & L \ar[r]^{g} \ar@{=}[d] & X \ar[d]^{\varepsilon}\\
Y \ar[r]^{\overline{f}}& L\ar[r]^{\overline{g}}&X.}$$
Denote by $V^L_{XY}$ the quotient set of $W^L_{XY}$ by the group $Aut(X) \times
Aut(Y)$. Since $f$ is an inflation and
$g$ is a deflation, this action is free. So
$$F^L_{XY}:=|V^L_{XY}|=\frac{|W^L_{XY}|}{a_Xa_Y},$$
where $a_X=|Aut(X)|$. The Hall algebra $\mathbf H
(\mathcal{A})$ is defined as the free $\mathbb{Z}$-module on the set of isomorphism classes of objects.
By abuse of a notation,
we write $X$ for the isomorphism classes $[X]$, and use the
numbers $F^L_{XY}$ as the structure constants of multiplication.
Define
$$X\circ Y:= \sum_{L}F^L_{XY}L,$$

\begin{prop}[\cite{H1}]\label{thm3.1} The Hall algebra $\mathbf H(\mathcal{A})$ of
a finitary and small exact category $\mathcal{A}$ is an associative
unital algebra.
\end{prop}

We now assume that $\mathcal{A}=Rep^f_R(\Gamma)$.
 For $\nu=(a_i)_{i \in I}, \tau=(b_i)_{i \in I}$,
 we define a bilinear form on $\mathbb{N}[I]$ by
\begin{equation}\label{eq3.1.1}
\langle \nu, \tau \rangle:=\sum_{i\in I}a_ib_i-\sum_{h\in H}a_{h'}b_{h''}.
\end{equation}

For any $X,Y \in Rep^f_R(\Gamma)$, we define a  multiplication by
\begin{equation}\label{mult}
XY:=q^{n\langle |X|, |Y| \rangle}X \circ Y.
\end{equation} 

\begin{prop}\label{thm3.1.1}
$\mathbf H(Rep^f_R(\Gamma))$ equipped with the  multiplication in (\ref{mult}) is an associative unital algebra.
\end{prop}
\begin{proof}
Proposition follows the bilinearity of $\langle-,-\rangle$.
\end{proof}

\subsection{The coalgebra structure}\label{sec3.3.2} Let $\Delta: \mathbf H(\mathcal{A})
\rightarrow \mathbf H(\mathcal{A}) \otimes
\mathbf H(\mathcal{A})$ be the map as following,
\begin{equation}\label{eq3.2}
  \Delta(E):= \sum_{M,N} \langle |M|,|N|\rangle F^E_{MN} \frac{a_M
a_N}{a_E} M \otimes N,
\end{equation}
where $M,N$ run through all conflations $M \rightarrowtail E
\twoheadrightarrow N$.
If one defines the
twisted multiplication on $\mathbf H(\mathcal{A}) \otimes
\mathbf H(\mathcal{A})$ to be
\begin{equation}\label{eq3.3}
  (A \otimes B) \cdot (C \otimes D):=q^{\frac{n}{2}(\langle |B|,|C|\rangle + \langle |C|,|B|\rangle)}  AC
\otimes BD,
\end{equation}
then as Green shows, in \cite{Gr}, the map $\Delta$ defined in (\ref{eq3.2}) is an
algebra homomorphism with respect to a twisted multiplication
on $\mathbf H(\mathcal{A}) \otimes \mathbf H(\mathcal{A})$ when
$\mathcal{A}$ is a hereditary abelian category, i.e., $\Delta$ gives a coalgebra structure on $\mathbf H(\mathcal{A})$.
Howerver, $\Delta$ is not a
homomorphism of algebras if $\mathcal{A}$ is an exact category.

 The following counterexample shows that  $\Delta:
\mathbf H(\mathcal{A}) \rightarrow \mathbf H(\mathcal{A})
\otimes \mathbf H(\mathcal{A})$, defined in (\ref{eq3.2}), cannot be an
algebra homomorphism under any twist in the case of $\mathcal{A}=Rep_R^f(\Gamma)$.

\begin{ex}\label{ex3.1}
 Let $\Gamma=A_2: 1 \rightarrow 2$, $R=k[t]/(t^n)\ (n>2)$, and $M=N=[R \xrightarrow{t} R]$.

If $\Delta$ is an algebra homomorphism, we must have 
\begin{equation}\label{eq3.4}
  \Delta(MN)=\Delta(M)\Delta(N).
\end{equation}

 On the right hand side of (\ref{eq3.4}), we consider the following diagram
  $$\xymatrix{
D\  \ar@{>->}[r]\  \ar@{>->}[d] & X\  \ar@{->>}[r] & B\  \ar@{>->}[d] \\
M\ \ar@{->>}[d] && N\ \ar@{->>}[d] \\
C \ \ar@{>->}[r] & Y \ \ar@{->>}[r] & A,}$$
where all possible choices for $B$ and $D$ are $[0
\xrightarrow {0} R]$, $[R \xrightarrow{t} R]$, and $[0 \xrightarrow{0}
0]$. Thus, all possible choices for $X$ are $[0 \xrightarrow{0} R^2]$, $[0
\xrightarrow {0} R]$,
$[0 \xrightarrow{0} 0]$, $[R \xrightarrow{t} R]$, \([R
\xrightarrow{\begin{bmatrix} t \\0 \end{bmatrix}} R^2]\),
and \( [R^2 \xrightarrow{
\begin{bmatrix}
  t & a\\0 & t
\end{bmatrix}}  R^2]\), where $a \in R$.

On the left hand side of (\ref{eq3.4}), we consider the following diagram
$$\xymatrix{& X\  \ar@{>->}[d]&
\\M\  \ar@{>->}[r] & E \ \ar@{->>}[r] \ \ar@{->>}[d]&N
\\&Y&}$$
Here \(E \simeq [R^2 \xrightarrow{\begin{bmatrix}t & a \\ 0 & t
\end{bmatrix}} R^2]\).
If $a \in tR$, then \[\begin{bmatrix}t&a\\0&t\end{bmatrix} \simeq \begin{bmatrix}t&0\\0&t\end{bmatrix}.\]  If $a$ is invertible in $R$, then \[\begin{bmatrix}t&a\\0&t\end{bmatrix} \simeq \begin{bmatrix}1&0\\0&t^2\end{bmatrix}.\]
Let \(E_1 \simeq [R^2 \xrightarrow{\begin{bmatrix}t & 0 \\ 0 & t
\end{bmatrix}} R^2]\) and \( E_2 \simeq [R^2 \xrightarrow{\begin{bmatrix}1 & 0 \\ 0 & t^2
\end{bmatrix}} R^2]\). Then,
$MN=b E_1 + c E_2$ for some nonzero numbers $b$ and $c$.
It is clear that $\Delta(E_2)$ has a summand $[R \xrightarrow{1} R]\otimes [R \xrightarrow{t^2} R]$.
This term, however, never appears on the
right hand side of (\ref{eq3.4}). This shows that $\Delta$
cannot be an algebra homomorphism.
\end{ex}

For $i\in I$, let $S_i$ be the unique projective representation of $\Gamma$ with $\underline{\dim}S_i=(\delta_{ij})_{j\in I}$.
The algebraic composition algebra associated to $\Gamma$ over $R$, denoted by $\mathbf C(R\Gamma)$, is the subalgebra of
 $\mathbf H(R\Gamma)$ generated by all $S_i$, for $i
\in I$.

\subsection{Relation between $\mathbf K$ and $\mathbf C(R\Gamma)$}

Let $\mathbf H(R\Gamma)^*$ be the dual Hall algebra of
$\mathbf{H}(R\Gamma)$, i.e., $\mathbf H(R\Gamma)^*=
\oplus_{\nu}\mathbf H(R\Gamma)_{\nu}^*$. Here
$\mathbf H(R\Gamma)_{\nu}^*$ is the set of all $\mathbb{C}$-valued
functions on the set of isomorphism classes of all representations
$M$ of $\Gamma$ over $R$ with dimension vector $|M| = \nu$. The
multiplication on $\mathbf H(R\Gamma)^*$ is defined as follows:
$$(f_1 \cdot f_2)(E)=
\sum_{N \subset E} f_1(E/N)f_2(N).$$ 
  Let $\mathbf C(R\Gamma)^*$ be the subalgebra of
$\mathbf H(R\Gamma)^*$ generated by $\delta_{S_i}, \forall i \in I$,
where $\delta_{S_i}$ is the characteristic function of $S_i$, i.e,
 $$\delta_{S_i}(x)=\left\{\begin{array}{ccc}1&& {\rm if}\ x=S_i\\0&&
{\rm others.}\end{array}\right.$$
By the following formular, $\mathbf C(R\Gamma)^*$ is isomorphic to the algebra $\mathbf C(R\Gamma)$.
$$(\delta_M \cdot \delta_N)(E)=\#\left\{L \subset E \ |\ L\simeq N, E/L \simeq M\right\}=F^E_{M,N}.$$

We define a map
\begin{align*}
  \chi: \mathbf K &\rightarrow \mathbf C(R\Gamma)^*\\
  A &\mapsto \chi_A,
\end{align*}
where
 $\chi_A(x)$
 is defined in Section \ref{sec2.5}.
\begin{prop}\label{thm3.5}
 The map $\chi$ is a surjective algebra homomorphism.
\end{prop}
\begin{proof}
By Theorem 4.1(b) in \cite{Lin},
  $\chi$ is an algebra homomorphism.
  We now show that $\chi$ is a surjective map.
 It is enough to show $\chi_{L_i}=\delta_{S_i}$ for any $\delta_{S_i}$.
In fact, $$\chi_{L_i}(x)=\chi_{\pi_{i!}{\bf 1}}(x)=\sum_{y\in \pi_i^{-1}(x)}\chi_{\bf 1}(y)=\chi_{\bf 1}(\pi_i^{-1}(x))=\delta_{S_i}.$$

    Here $\pi_i$ is the obvious projection map. The penultimate equality is true because both $E_V$
     and $\widetilde{\mathcal{F}}_V$ contain a single point.
\end{proof}

Denote $\mathcal{I}_1={\rm Ker}(\chi)$. Then $\mathbf C(R\Gamma) \simeq
\mathbf K/\mathcal{I}_1$.

We note that, for the case that $n=1$, i.e., $R$ is a field, Lusztig and Ringel show that
the ideal $\mcal I_1$ is  generated by the quantum Serre relations.
However, for
the case that $n\geq 2$, the ideal
$\mathcal{I}_1$ is  more complicated.
The following example gives some ideas of what
$\mathcal{I}_1$ is.

\begin{ex}\label{ex3.2}
  Fix $R=\mathbb{F}_q[t]/(t^n)$, consider quiver $A_2: 1 \rightarrow 2$. Then $S_1: [R \rightarrow
  0]$ and $S_2: [0 \rightarrow R]$ are all simple objects. By computation, we have
\begin{equation*}
  \begin{split}
  &S^2_1=q^{n/2}(q^n+q^{n-1})[ R^2 \rightarrow 0],\\
&S_1S_2=q^{-n/2}([R \xrightarrow{0} R]+[R \xrightarrow{1}
  R]+[R \xrightarrow{t} R]+\cdots +[R \xrightarrow{t^{n-1}} R]),\\
&S_2S_1=[R \xrightarrow{0} R]\\
&S^2_1S_2=q^{-n/2}(q^n+q^{n-1})([R^2 \xrightarrow{0} R]+[R^2 \xrightarrow{(1,0)}
  R]+[R^2 \xrightarrow{(t,0)} R]+\cdots +[R^2 \xrightarrow{(t^{n-1},0)} R]),\\
&S_1S_2S_1=(q^n+q^{n-1})[R^2 \xrightarrow{0} R]+[R^2 \xrightarrow{(1,0)}
  R]+q[R^2 \xrightarrow{(t,0)} R]+\cdots +q^{n-1}[R^2 \xrightarrow{(t^{n-1},0)} R],\\
&S_2S^2_1=q^{n/2}(q^n+q^{n-1})[R^2 \xrightarrow{0} R].
 \end{split}
\end{equation*}
There is no 
quantum Serre relation at this time.

\end{ex}

\section{Quantum generalized Kac-Moody algebras} 

\subsection{} 

Let $I$ be a countable index set. A symmetric {\it generalized
root datum} (see \cite{KS}) is a matrix $A=(a_{ij})_{i,j\in I}$ satisfying the
following conditions:
\begin{itemize}
\item[(a)] $a_{ii} \in \left\{2,0,-2,-4, \cdots \right\}$, and
\item[(b)] $a_{ij}=a_{ji} \in \mathbb{Z}_{\leq 0}.$
\end{itemize}
Such a matrix is a special case of Borcherds-Cartan matrix. Let
$I^{re}=\left\{i \in I\ |\ a_{ii}=2\right\}$ and $I^{im}=I\setminus
{I^{re}}$. A collection of positive integers ${m}=(m_i)_{i \in I}$
with $m_i=1$ whenever $i\in I^{re}$ is called the charge of $A$.

\subsection{} 

The {\it quantum generalized Kac-Moody algebra} (see \cite{KS}) associated with
$(A,m)$ is the $\mathbb{Q}(v)$-algebra $\mbf U_v(\mathfrak{g}_{A,m})$
generated by the elements $K_i,K^{-1}_i,$ $E_{i,k}$, and $F_{i,k}$
for $i \in I, k=1, \cdots ,m_i$ subject to the following relations:
\begin{equation*}
 K_iK^{-1}_i=K^{-1}_iK_i=1,\  K_iK_j=K_jK_i,
\end{equation*}
\begin{equation*}
K_iE_{jk}K^{-1}_i=v^{a_{ij}}E_{jk},\
K_iF_{jk}K^{-1}_i=v^{-a_{ij}}F_{jk},
\end{equation*}
\begin{equation*}
E_{ik}F_{jl}-F_{jl}E_{ik}=\delta_{lk}\delta{ij}
\frac{K_i-K^{-1}_i}{v-v^{-1}},
\end{equation*}
\begin{equation*}
 \sum_{n=0}^{1-a_{ij}}(-1)^n \left[
{{1-a_{ij}}\atop n}\right]E^{1-a_{ij}-n}_{ik}E_{jl}E^n_{ik}=0,
\forall i\in I^{re}, j\in I, i\neq j,
\end{equation*}
\begin{equation*}
 \sum_{n=0}^{1-a_{ij}}(-1)^n \left[
{{1-a_{ij}}\atop n}\right]F^{1-a_{ij}-n}_{ik}F_{jl}F^n_{ik}=0,
\forall i\in I^{re}, j\in I, i\neq j, and
\end{equation*}
\begin{equation*}
E_{ik}E_{jl}-E_{jl}E_{ik}=F_{ik}F_{jl}-F_{jl}F_{ik}=0, \ \rm{if} \
\it{a}_{ij}=\rm{0}.
\end{equation*}
Here $\left[n \atop k \right]=\frac{[n]!}{[n-k]![k]!},
[n]!=\prod^n_{i=1}[i]$, and $[n]=\frac{v^n-v^{-n}}{v-v^{-1}}.$

In this paper, we only consider the case in which $I=I^{im}$ and all $m_i=1$.
 Under this assumption, we shall simply write $E_i$ (resp. $F_i$) instead of $E_{ik}$ (resp. $F_{ik}$).

\subsection{Relation between $\mathbf K$ and $\mbf U^-_v$}

Define a bilinear form on $\mathbf K$ as follows,
$$(L,M)_{\mathbf K}=\sum_jd_j(E_V,G_V;L,M)v^{-j}\quad \text{for all semisimple complexes}\ L, M.$$
\begin{prop}\label{prop3.11}
The bilinear form $(-,-)_{\mathbf K}$ defined above is non-degenerate.
\end{prop}
\begin{proof}

Firstly, by the properties of $d_j(E,G;L,M)$ in Section \ref{section4.4}, this
is a bilinear form.
Secondly, by Theorem \ref{thm3.3}, all simple
perverse sheaves in
$\mathcal{P}^f_{\nu}$ for various $\nu$ form a
$\mathbb{Z}[v,v^{-1}]$-basis of $\mathbf K$. So, for any $L \in
\mathbf K$, $L$ can be written into $L=\sum_K
C_KK$
for $C_K \in \mathbb{Z}[v,v^{-1}]$ and all $K$ are simple perverse sheaves.
Now for any  $L \in
\mathbf K$, let $M$ be a simple direct summand of $L$. By Property
(d) of $d_j(E,G;L.M)$ in Section \ref{section4.4}, $(L,
M)_{\mathbf K} \neq 0$. Hence the bilinear form is
non-degenerate.
\end{proof}

Let $\mbf U^-_v$ be the $\mathbb{Z}[v,v^{-1}]$-subalgebra of $\mbf U_v(\mathfrak{g}_{A,m})$ generated by all
$F_{i}$.
 $\mbf U_v^-$ is an $\mbb N[I]$-graded algebra by setting $\deg F_i=i$.
 It is clear that $\mbf U^-_v$ only subjects to the relation $F_iF_j=F_jF_i$ for all $i, j\in I$ with $a_{ij}=0$.

We define a map
\begin{align*}
  f: \mbf U^-_v &\rightarrow \mathbf K\\
 F_i &\mapsto L_{i,1}. 
\end{align*}
It is easy to check that
$f(F_iF_j)=f(F_jF_i)$ if $a_{ij}=0$.
 So this map can be extended to an algebra
homomorphism.
In addition, $f$ preserves the grading, where the
grading of $B\in \mathbf K$ is defined as $\tau$
when $B$ is  in $\mbf K_{\tau}$.
 Now
define a bilinear form $(-,-)_{\mbf U}$ on $\mbf U^-_v$ as
$$(A,B)_{\mbf U}=(f(A),f(B))_{\mathbf K}.$$
\begin{prop}\label{thm3.4}
${\rm Ker}(f)=Rad(-,-)_{\mbf U}=:\mathcal{I}_2$, so $\mbf U^-_v/\mathcal{I}_2 \simeq
\mathbf K.$
\end{prop}
\begin{proof}
Obviously, ${\rm Ker}(f) \subset \mathcal{I}_2$.

Let us pick any $x \in \mathcal{I}_2$. Then for any $y \in
\mathbf K$, there exists $z \in \mbf U^-_v$ such that $f(z)=y$ due to
the fact that $f$ is a surjective map. Therefore,
$$0=(x,z)_{\mbf U}=(f(x),f(z))_{\mathbf K}=(f(x),y)_{\mathbf K}.$$
This implies that $f(x) \in Rad(-,-)_{\mathbf K}$. Since the bilinear
form $(-,-)_{\mathbf K}$ is non-degenerate, $f(x)=0$, i.e., $x
\in {\rm Ker}(f)$.
Hence $\mbf U^-_v/\mathcal{I}_2 \simeq \mathbf K$.
\end{proof}


\begin{thebibliography}{10}


\bibitem{BBD}
A.~A. Be{\u\i}linson, J.~Bernstein, and P.~Deligne, \emph{Faisceaux pervers},
  Analysis and topology on singular spaces, {I} ({L}uminy, 1981), Ast\'erisque,
  vol. 100, Soc. Math. France, Paris, 1982, pp.~5--171.

\bibitem{craw1992}
W.~C. Boevey, \emph{{Lectures on representations of quivers}}, Preprint http://www1.maths.leeds.ac.uk/~pmtwc /quivlecs.pdf (1992).

\bibitem{Braden}
T.~Braden, \emph{{Hyperbolic localization of intersection cohomology}},
  Transform. Groups \textbf{8} (2003), no.~3, 209--216.

\bibitem{Fan0}
	Z.~Fan, \emph{{Geometric approach to Hall algebras and character sheaves}}, Thesis (Ph.D.)--Kansas State University, 2012,  ProQuest LLC.

\bibitem{FK}
E.~Freitag and R.~Kiehl, \emph{\'{E}tale cohomology and the {W}eil conjecture},
  Ergebnisse der Mathematik und ihrer Grenzgebiete (3).

\bibitem{FS}
L.~Frerick and D.~Sieg, \emph{{Exact Categories in Functional Analysis}},
  Preprint http://www.math.uni-trier.de/abteilung/analysis/HomAlg.pdf.


\bibitem{Gr}
J.~A. Green, \emph{Hall algebras, hereditary algebras and quantum groups},
  Invent. Math. \textbf{120} (1995), no.~2, 361--377.

\bibitem{GL}
I.~Grojnowski and G.~Lusztig, \emph{A comparison of bases of quantized
  enveloping algebras}, Linear algebraic groups and their representations
  ({L}os {A}ngeles, {CA}, 1992), Contemp. Math., vol. 153, Amer. Math. Soc.,
  Providence, RI, 1993, pp.~11--19.

\bibitem{H1}
A.~Hubery, \emph{From triangulated categories to {L}ie algebras: a theorem of
  {P}eng and {X}iao}, Trends in representation theory of algebras and related
  topics, Contemp. Math., vol. 406, Amer. Math. Soc., Providence, RI, 2006,
  pp.~51--66.

\bibitem{KS}
S.~J. Kang and O.~Schiffmann, \emph{Canonical bases for quantum generalized
  {K}ac-{M}oody algebras}, Adv. Math. \textbf{200} (2006), no.~2, 455--478.

\bibitem{KSV}
M.~Kapranov, O.~Schiffmann and E.~Vasserot,  \emph{The spherical Hall algebra of Spec($\mbb Z$)},
arXiv preprint arXiv:1202.4073(2012).


\bibitem{Li}
Y.~Li, \emph{Affine canonical bases}, Thesis
  (Ph.D.)--Kansas State University, 2006, ProQuest LLC.

\bibitem{LZ}
 Y.~Li and Z.~Lin, \emph{A realization of quantum groups via product valued quivers},
 Algebr. Represent. Theory \textbf{13} (2010), no.~4, 427-444.

\bibitem{Lin}
Z.~Lin, \emph{Lusztig's geometric approach to {H}all algebras}, Representations
  of finite dimensional algebras and related topics in {L}ie theory and
  geometry, Fields Inst. Commun., vol.~40, Amer. Math. Soc., Providence, RI,
  2004, pp.~349--364.

\bibitem{L9}
G.~Lusztig, \emph{Cuspidal local systems and graded {H}ecke algebras. {I}},
  Inst. Hautes \'Etudes Sci. Publ. Math. (1988), no.~67, 145--202.

\bibitem{L1}
G.~Lusztig, \emph{Canonical bases arising from quantized enveloping algebras},
  J. Amer. Math. Soc. \textbf{3} (1990), no.~2, 447--498.

\bibitem{L2}
G.~Lusztig, \emph{Quivers, perverse sheaves, and quantized enveloping
  algebras}, J. Amer. Math. Soc. \textbf{4} (1991), no.~2, 365--421.

\bibitem{L5}
G.~Lusztig, \emph{Introduction to quantum groups}, Modern Birkh\"auser
  Classics, Birkh\"auser/Springer, New York, 2010.

\bibitem{Mus}
M.~Musta{\c{t}}{\u{a}}, \emph{Jet schemes of locally complete intersection
  canonical singularities}, Invent. Math. \textbf{145} (2001), no.~3, 397--424.

\bibitem{Q}
D.~Quillen, \emph{Higher algebraic {$K$}-theory. {I}}, Algebraic {$K$}-theory,
  {I}: {H}igher {$K$}-theories ({P}roc. {C}onf., {B}attelle {M}emorial {I}nst.,
  {S}eattle, {W}ash., 1972), Springer, Berlin, 1973, pp.~85--147. Lecture Notes
  in Math., Vol. 341.


\bibitem{X1}
J.~Xiao, \emph{Drinfeld double and {R}ingel-{G}reen theory of {H}all algebras},
  J. Algebra \textbf{190} (1997), no.~1, 100--144.

\end{thebibliography}

\vspace{10pt}
  {Zhaobing Fan$^{1,2}$\\
1. Mathematics department,\ Kansas State University, \ U.S. \ 66506
 \\2. School of science,\ Harbin Engineering University,\ China,\
 150001\\
 \hspace{20pt}  \email{fanz@math.ksu.edu}}

\end{document}